\documentclass[11pt]{article}

\author{
Patrick Dondl
\thanks{Abteilung f{\"u}r Angewandte Mathematik, 
Albert-Ludwigs-Universit{\"a}t Freiburg, 
Hermann-Herder-Str.~10, 
79104 Freiburg, Germany; 
{\tt patrick.dondl@mathematik.uni-freiburg.de} }, 
Martin Jesenko
\thanks{Abteilung f{\"u}r Angewandte Mathematik, 
Albert-Ludwigs-Universit{\"a}t Freiburg, 
Hermann-Herder-Str.~10, 
79104 Freiburg, Germany; 
{\tt martin.jesenko@mathematik.uni-freiburg.de} }, 
Michael Scheutzow
\thanks{
Institut f\"{u}r Mathematik, 
MA 7--5, 
Fakult\"{a}t II,
Technische Universit\"{a}t Berlin, 
Stra{\ss}e des 17.~Juni~136,
10623 Berlin, 
Germany;
{\tt ms@math.tu-berlin.de}
}
}

\title{Infinite pinning}

\usepackage[a4paper]{geometry}
\usepackage{enumerate}
\usepackage{mathrsfs}
\usepackage{pifont,stmaryrd}
\usepackage{fancyhdr}
\usepackage{amssymb, amsmath, amstext, amsthm}
\usepackage{titling}
\usepackage{eucal}
\usepackage{aurical}
\usepackage[T1]{fontenc}
\usepackage{pictexwd,dcpic}
\usepackage{graphicx}
\usepackage{wrapfig, float}
\usepackage{todonotes}

\usepackage{pictexwd,dcpic}
\usepackage{graphicx}

\newtheorem{theo}{Theorem}[section] 
\newtheorem{lemma}[theo]{Lemma}

\newtheorem{prop}[theo]{Proposition}

\theoremstyle{definition}

\theoremstyle{remark}
\newtheorem{remark}[theo]{Remark}

\def\F {{\cal F}}

\def\x {\ dx}

\def\R {\mathbb R}
\def\E {\mathbb E}
\def\P {\mathbb P}

\def\N {\mathbb N}
\def\Z {\mathbb Z}

\def\i {\infty}

\def\* {$*$}

\usepackage{color}

\begin{document}

\maketitle

\begin{abstract}
In this work, we address the occurrence of infinite pinning in a random medium. 
We suppose that an initially flat interface starts to move through the medium due to some constant driving force. The medium is assumed to contain random obstacles. We model their positions by a Poisson point process and their strengths are not bounded.
We determine a necessary condition on its distribution so that regardless of the driving force the interface gets pinned.



\end{abstract}
\section{Introduction}

In this work we consider two models for an interface propagation that were already presented in \cite{DJS}. There it is shown that pinning occurs even in a quenched heterogeneous medium that has -- when spatially averaged -- no influence on the propagation of an interface. For more information on the setting and the models, including examples from the physics literature, we refer to \cite{DJS}.

In the present work, we consider the question whether \emph{infinite pinning} can occur, i.e., the interface becomes stuck for any arbitrarily large driving force, and for which distributions of obstacles this takes place. In the 1+1-dimensional setting we show that this can happen even if the expectation of the obstacle strength is finite.

Generally, we prove pinning by establishing the existence of a non-negative, stationary supersolution to the evolution equation for the interface. 

A canonical version of our model for the propagation of an interface in a heterogeneous medium is given by a semilinear parabolic equation of the form
\[ \partial_{t} u(t,x) = \Delta u(t,x) - f(x,u(t,x)) + F \]
In our setting of \emph{random} heterogeneous media, $f$ is a given quenched, random field.

As initial condition we consider $ u(0,\cdot) = 0 $. The basic idea is to a.s.~find a (viscosity) supersolution (for the definition and properties see, e.g., \cite{Crandall:1992kn}) to the related stationary problem, i.e., a function $v$ that satisfies
\[ \Delta v(x) - f(x,v(x)) + F \le 0
\quad \mbox{and} \quad
v(x) \ge 0
\quad \mbox{for all } x \in \R.  \] 
By employing an appropriate comparison principle, this immediately implies that the interface a.s. stays below the graph of $v$ for all times since this was the case at $ t = 0 $. Our main goal in this work is thus to show the existence of such a non-negative stationary supersolution in the setting of our two models.
 
The first model we study is a spatially purely discrete variant of the above. We there consider the lattice $\Z^{2}$. Each lattice point acts with a force of random strength chosen by a suitable probability distribution. The notions of the space and time derivative are adapted to the discrete case. This model was studied in, e.g., \cite{Bodineau:2013ur}, where for some very specific distributions of $f$ pinning and depinning results were shown. Here, we focus on results regarding pinning, but consider a large class of distributions (in particular, allowing the aforementioned case of $f$ having zero mean).

\section{Discrete model on $\Z^2$}

In the discrete model on $\Z^2$, the shape of the interface is determined by a function $ \Z \to \Z $. Its propagation is therefore given as a time evolution of (random) functions $u_t: \Z \to \Z$, $t \ge 0$, with the initial condition $ u_0\equiv 0 $\footnote{Note that, as usual in such discrete settings, the subscript does not indicate a partial derivative. Here $u_t$ is the state of the interface at time $t$.}. At any time $t$, the function $u$ may jump from its current value $u_t(i)$ only to $u_t(i) \pm 1$ depending on the current jump rate $ \lambda $. For $\lambda >0$, $u$ can only jump to $u_t(i)+1$ with rate $\lambda$, whereas for $\lambda <0$, $u$ can only jump to $u_t(i)-1$ and does so with rate $-\lambda$. The jump rate depends on the local shape of the interface and the obstacle force at the current position. More precisely,
\[\lambda=\Lambda\big(\Delta_{1} u_t(i)-f(i,u_t(i))\big) \]
where $\Delta_{1} u(i) = u(i+1) + u(i-1) - 2 u(i) $ is the discrete Laplacian, $f(i,j)$ is the obstacle strength at $(i,j) \in \Z^{2}$ and $\Lambda$ is a strictly increasing and bounded function from $\Z$ to $\R$ which satisfies $\Lambda(0)=0$. We suppose $f(i,j)$, $i,j \in \Z$, to be independent and identically distributed $\N_{0}$-valued random variables defined on a probability space $(\Omega,\F,\P)$. The strength 0 simply means the absence of an obstacle. In this article we only consider the case where $f(i,j)\ge 0$ for all $i,j \in \Z$, i.e., the forces are non-negative. For more details and the definition of jump rate consult \cite{DJS}.

We will show the following simple criterion for infinite pinning.

\begin{theo}
\label{theo:discrete-main}
Let for all $ (i,j) \in \Z^{2} $ be $ f(i,j) \sim X $ for some $ \N_{0} $-valued random variable $X$. 
If $X$ has unbounded second moment, i.e., the expectation $ \E( X^{2} ) = \i $, then infinite pinning occurs.
\end{theo}

In \cite{DJS}, see Collorary 2.3, it was shown that if $ X_{0} \sim X_{1} \sim \ldots \sim f(i,j) $ are independent random variable (possibly even with values in $ \Z $) for which 
\[ \E  \big( \sup\{X_0,-1+X_1, -2+X_2, \dots \} \big)>F \]
for some $F \in \Z$, then almost surely there exists a function  $v :\Z \to \N_0$ such that $\Delta_{1} v(i) \le f(i,v(i)) - F $. As mentioned in the introduction, this stationary supersolution acts as a barrier for the interface since the jump rate is in every point non-positive. Hence, a sufficient condition for infinite pinning reads

$$
\E \big( \sup \{  -j + X_j : j \in \N_{0} \} \big)=\infty.
$$ 
Thus, Theorem~\ref{theo:discrete-main} follows immediately from
\begin{lemma}
Let $X$ be a random variable with values in $ \N_{0} $. Then the following statements are equivalent
\begin{itemize}
\item 
For any independent random variables $ X_{0}, X_{1} , \ldots $ having the same distribution as $X$, we have
$ \E \big( \sup \{  -j + X_j : j \in \N_{0} \}  \big)=\infty $.
\item
$ \E(X^{2} ) = \i. $
\end{itemize}
\end{lemma}

\begin{proof}
Define
$$
M:= \sup \{  -j + X_j : j \in \N_{0} \}.
$$
Then, for $n \in \N$,
$$
\P\big(M \ge n)\le \P(X_0 \ge n)+\P(X_1 \ge n+1)+\P(X_2 \ge n+2)+...=\sum_{l=0}^\infty (l+1)\P(X=n+l)
$$
and therefore
$$
\E(M)
= \sum_{n=1}^\infty \P\big(M \ge n)
\le \sum_{l=0}^\infty (l+1) \sum_{n=1}^\infty  \P(X=n+l)
= \sum_{k=0}^\infty \Big(\P(X=k) \sum_{l=0}^k l\Big)
< \infty  
$$
in case $\E( X^2 ) <\infty$. 

On the other hand, assume $\E (X^2)=\infty$. If $\E X=\infty$, then $\E M \ge \E X_0=\infty$. Therefore, let us explore the case $\E X<\infty$.

Let $\alpha_k:=\P(X \ge k)$, $k \in \N$. Then for each $ n \in \N $
\[ \P(M \ge n) = 1 - \P(M < n) = 1 - \prod_{k=0}^{\infty} \P( X_{k} - k < n ) = 1 - \prod_{k=0}^{\infty} ( 1 - \alpha_{n+k} ) = 1 - \prod_{k=n}^{\infty} ( 1 - \alpha_{k} ). \]
Since $\E X<\infty$, there exists some $k_0 \in \N$ for which $\sum_{k=k_0}^{\infty} \alpha_k   \le \frac{1}{2} $. 
For every $y \in [0,\frac{1}{2}] $ it holds $ - y \ge \log (1-y) \ge -2 y$. Therefore, for $n \ge k_0$,
\begin{align*}
\P(M \ge n)
& =1-\prod_{k=n}^\infty(1-\alpha_k) \\
& \ge \frac 1{2}\Big(1-\prod_{k=n}^\infty(1-\alpha_k)^2 \Big)\\
& = \frac 1{2}\Big(1-\exp\big\{2 \sum_{k=n}^\infty \log \big(1-\alpha_k\big)   \big\}\Big)\\
& \ge \frac 1{2}\Big( 1-\exp\big\{-2 \sum_{k=n}^\infty \alpha_k   \big\}\Big)\\
& \ge \frac 1{2}\Big(   1-\exp\big\{\log \big(1-\sum_{k=n}^\infty \alpha_k   \big)               \big\}  \Big)\\
& = \frac 1{2}\sum_{k=n}^\infty \alpha_k , 
\end{align*}
where the first inequality follows since $1-u\ge \frac{1}{2}(1-u^2)$ for $u \in [0,1]$.
Therefore,
$$
\E M
= \sum_{n = 1}^{\infty} \P( M \ge n ) 
\ge \sum_{ n = k_0 }^{\infty} \P( M \ge n )
\ge \frac{1}{2} \sum_{ n = k_0 }^{\infty} \sum_{k=n}^{\infty} \alpha_k      
= \frac{1}{2} \sum_{k=k_0}^\infty(k-k_0+1) \alpha_k 
= \infty
$$
since
\[ \i = \E( X^{2} ) = \sum_{k=1}^{\i} (2k-1) \P(X \ge k) = 2 \sum_{k=1}^{\i} k \alpha_{k} - \E X. \]

\end{proof}

\section{Continuous model}
For the continuous model, we take the setting described in \cite{DDS} that we now shortly present.
We are investigating the behavior of solutions $ u \colon \R^n\times [0,\infty) \times \Omega \to \R $ of the semilinear parabolic partial differential equation
\begin{eqnarray} 
\Delta u( x , t , \omega ) - f( x , u( x , t , \omega ) , \omega ) + F & = & \partial_{t} u( x , t , \omega ) , \label{eq:model} \\
u(x,0,\omega) & = & 0. \nonumber
\end{eqnarray}
%

%
For the sake of simplicity, we suppose that all the obstacles are of the same shape and have the following properties:
\begin{itemize}
\item
Shape of obstacles is given by the function $ \varphi \in C_{c}^{\i}( \R^{n} \times \R ) $ that satisfies
\[ \varphi(x,y) \ge 1 \mbox{ for } \max \{ |x|,|y| \} \le r_{0}
\quad \mbox{and} \quad
\varphi(x,y) = 0 \mbox{ for } | (x, y) | \ge r_{1} \]
for some $ r_{0}, r_{1} > 0 $ with $ r_{1} > \sqrt{n} r_{0} $.
\item 
Obstacle positions $ \{ ( x_{i} , y_{i} ) \}_{ i \in \N } $ are distributed
according to an $ (n + 1) $-dimen\-sional Poisson point process on $ \R^{n} \times [ r_{1} , \i ) $ 
with intensity $ \lambda > 0 $.
\item 
Obstacle strengths $ \{ f_{i} \}_{ i \in \N } $ are independent and identically distributed strictly positive random variables
($ f_{i} \sim f_{1} $ for all $ i \in \N $) that are independent of $ \{ ( x_{i} , y_{i} ) \}_{ i \in \N } $.
\end{itemize}

Thus, the force of the obstacle field is the random function
\[ f( x , y, \omega ) = \sum_{i} f_{i}( \omega ) \varphi( x - x_{i}( \omega ), y - y_{i}( \omega ) ). \]

\begin{remark}
We note that this specific form is only an example which we focus on for concreteness' sake and to simplify the exposition. Variants of these obstacle distributions, e.g., obstacles centered on lattice sites with random strength, lead to the same results.
\end{remark}

%

Again, our goal is to construct a non-negative, stationary supersolution for all $F>0$.
The construction heavily relies on the ideas in \cite{DDS}.
%
%


The first step of the construction is to find a stationary supersolution for a related Neumann problem on a ball with an obstacle in its center.

From the requirement of being a supersolution, it is clear that the Laplacian of our constructed function may be positive (if $f$ is sufficiently large there)  inside an obstacle and it must be negative (below $-F$) outside, in order to compensate for the driving force $F$.

We thus explicitly construct a function that is radially increasing, has an appropriate Laplacian and becomes flat at the boundary of the ball, see Figure~\ref{figure:ena-ovira}. 

\begin{figure}[ht]
\begin{center}
\includegraphics[width=.9\textwidth]{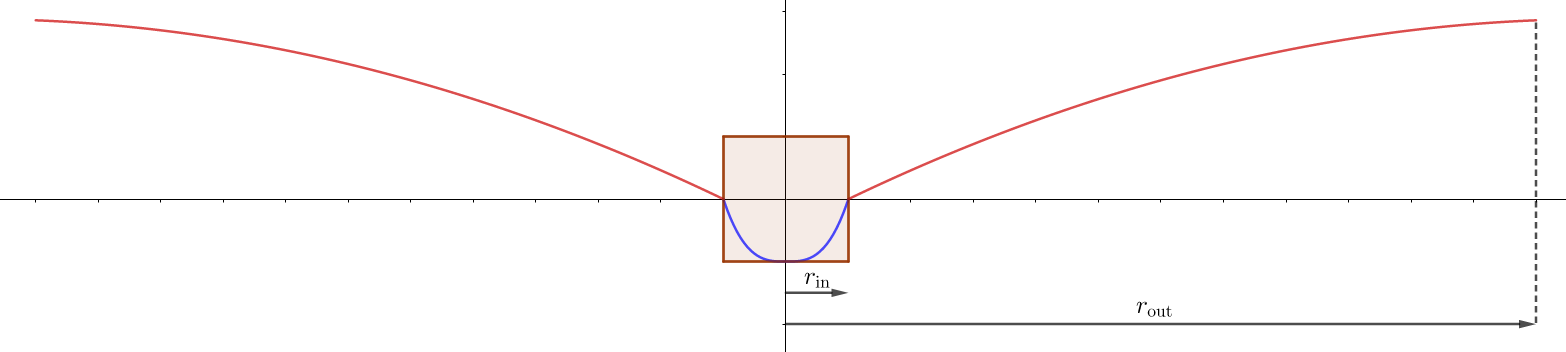}
\end{center}
\vspace{-5mm}
\caption{Construction of the local supersolution. 
We construct a radial function with a given positive Laplacian in the inner ball and a negative constant Laplacian on the outer ring.
It should get flat a the boundary and may be non-differentiable as long as the derivative jumps downwards.}
\label{figure:ena-ovira}
\end{figure}
More precisely, we choose some $ 0 < r_{\rm in} < r_{\rm out} $ and $ F_{\rm in}, F_{\rm out} > 0 $.
The parameter $ r_{\rm in} $ determines the cylinder 
$ \{ (x,y) \in \R^{n} \times \R: |x| \le r_{\rm in}, |y| \le r_{\rm in} \} $
where we suppose the obstacle to have the full strength.
We are looking for a radially-symmetric function $ v_{\rm local} $ that satisfies 

\begin{equation} 
\label{eq:local-sol}
\Delta v_{\rm local}(x) \le 
\begin{cases} 
F_{\rm in}, & |x| < r_{\rm in}, \\ 
- F_{\rm out}, & r_{\rm in} < |x| < r_{\rm out}.
\end{cases}
\end{equation}
It may be non-differentiable on $ \partial B_{\rm in} $, however, to be a viscosity supersolution, it must fulfil
\[ \lim_{ \genfrac{}{}{0pt}{}{ x \in B_{\rm in} }{ x \to x_{0} } } \partial_{r} v_{\rm local}(x) 
\ge \lim_{ \genfrac{}{}{0pt}{}{ x \not\in \overline{ B_{\rm in} } }{ x \to x_{0} } } \partial_{r} v_{\rm local}(x)
\quad \mbox{for every } x_{0} \in \partial B_{\rm in}. \]
We impose 
\[ | v_{\rm local}(x) | \le r_{\rm in} \quad \mbox{for } |x| < r_{\rm in} \]
since the solution must lie within the cylinder modelling an obstacle.
Moreover, let
\[ \partial_{r} v_{\rm local}(x) = 0 \quad \mbox{for } |x| = r_{\rm out} \] 
and $ v_{\rm local}(x) = \i $ if $ |x| > r_{\rm out} $.  
Denote the (blue) part in the inner ball by $ v_{ \rm in } $.
Let $ m \in \N $ be arbitrary. 
(It will serve to ensure that the supersolution stays in the obstacle.)
In contrast to \cite{DDS}, we do not (necessarily) take constant Laplacian but allow for a specific function instead. We choose a radially-symmetric function $v_{ \rm in }(x)=\phi(|x|)$ with Laplacian $\Delta v_{ \rm in }(x) = F_{\rm in}\cdot( \frac{ |x| }{ r_{\rm in} } )^{m} $, i.e., 
\[ \phi''(r) + \frac{n-1}{r} \phi'(r) = F_{\rm in}\cdot \left( \frac{ r }{ r_{\rm in} } \right)^{m} 
\quad \mbox{with} \quad 
\phi( r_{\rm in} ) = 0. \]
Then
%
%
%
%
\[ \phi(r) = \frac{ F_{\rm in} }{ (m+n)(m+2) r_{\rm in}^{m} } 
( r^{m+2}- r_{\rm in}^{m+2} ). \]
To stay in the cylinder, it must hold that
\[ \phi(0) = \frac{ - F_{\rm in} r_{\rm in}^{2} }{ (m+n)(m+2) } \ge -r_{ \rm in }. \]
The derivative at the boundary is
\[ \phi'( r_{ \rm in } ) = \frac{ F_{\rm in} r_{\rm in} }{ m+n }. \]
In the remaining part of the ball, we take the same function $ v_{\rm out} $ as in \cite{DDS}. 
Again it is radially symmetric, and therefore we write
$ v_{\rm out}(x) = \psi(|x|) $.
It should meet $ v_{\rm in} $ at $ r_{\rm in} $, 
i.e.~$ \psi( r_{\rm in} ) = 0 $, and has a zero normal derivative on $ \partial B_{ r_{\rm out } } $, therefore, 
$ \psi'( r_{\rm out} ) = 0 $. 
Its Laplacian is simply $ - F_{ \rm out } $. Hence,
\[ \psi''(r) + \frac{n-1}{r} \psi'(r) = - F_{\rm out}. \]
For our construction, the function values of $\psi$ are irrelevant -- it is enough to consider its derivative. 
%
%
%
We obtain
\[ \psi'(r)= \frac{ F_{ \rm out } }{ n } \frac{ ( r_{ \rm out }^{n} - r^{n} ) }{ r^{n-1} }. \]
Hence,
\[ \psi'( r_{ \rm in } ) 
= \frac{ F_{ \rm out } }{ n r_{ \rm in }^{n-1} } ( r_{ \rm out }^{n} - r_{ \rm in }^{n} ) 
= \frac{ F_{ \rm out } r_{ \rm in } }{ n } \left( \frac{ r_{ \rm out }^{n} }{ r_{ \rm in }^{n} } - 1 \right). \]
If we define $ v_{ \rm local } $ in such a way, it will be a viscosity supersolution if 
\begin{equation}
\label{eq:cond-loc-subsol}
\phi'( r_{ \rm in } ) \ge \psi'( r_{ \rm in } )
\quad \mbox{or} \quad
\frac{ F_{\rm in} }{ m+n } \ge \frac{ F_{ \rm out } }{ n }  \left( \frac{ r_{ \rm out }^{n} }{ r_{ \rm in }^{n} } - 1 \right). 
\end{equation}

If we take the minimum of appropriately translated local supersolutions, as depicted in Figure~\ref{figure:vec-ovir}, 
we obtain a supersolution for a problem with obstacles all having the height-coordinate $ y $ equal 0.
We call this function the \emph{flat supersolution}.

\begin{figure}[ht]
\begin{center}
\includegraphics[width=\textwidth]{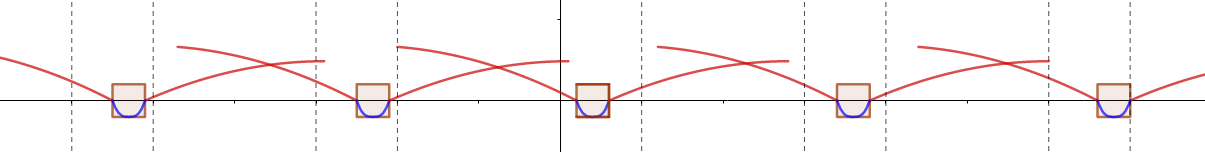}
\end{center}
\vspace{-5mm}
\caption{Flat supersolution. 
If we have several translated local supersolutions such that their domains cover $ \R^{n} $, their minimum is a viscosity supersolution of the corresponding equation.}
\label{figure:vec-ovir}
\end{figure}

Since we are in the random case, we must first localize sufficiently many obstacles, i.e., find a.s.~an array of them.
Its existence will follow from the result from the percolation theory below.

\begin{theo}[{\cite[Theorem 1]{DDGHS}}] 
\label{theo:percolation}
Suppose that to each  $ z \in \Z^{n+1} $ a state is assigned that can be open or closed.
For every point the probability that it is open is $ p \in (0,1) $ with different points receiving independent states. 
If $ p > 1 - \frac{1}{ (2n + 2)^{2} } $, 
then there exists a.s. a (random) function $ L : \Z^{n} \to \N $ with the following properties:
\begin{itemize}
\item
For each $ a \in \Z^{n} $, the site $ (a, L(a)) \in \Z^{n+1} $ is open.
\item
For any $ a, b \in \Z^{n} $ with $ \| a - b \|_{1} = 1 $, we have $ | L(a) - L(b) | \le 1 $.
\end{itemize}
\end{theo}



\begin{figure}[ht]
\begin{center}
\includegraphics[width=.5\textwidth]{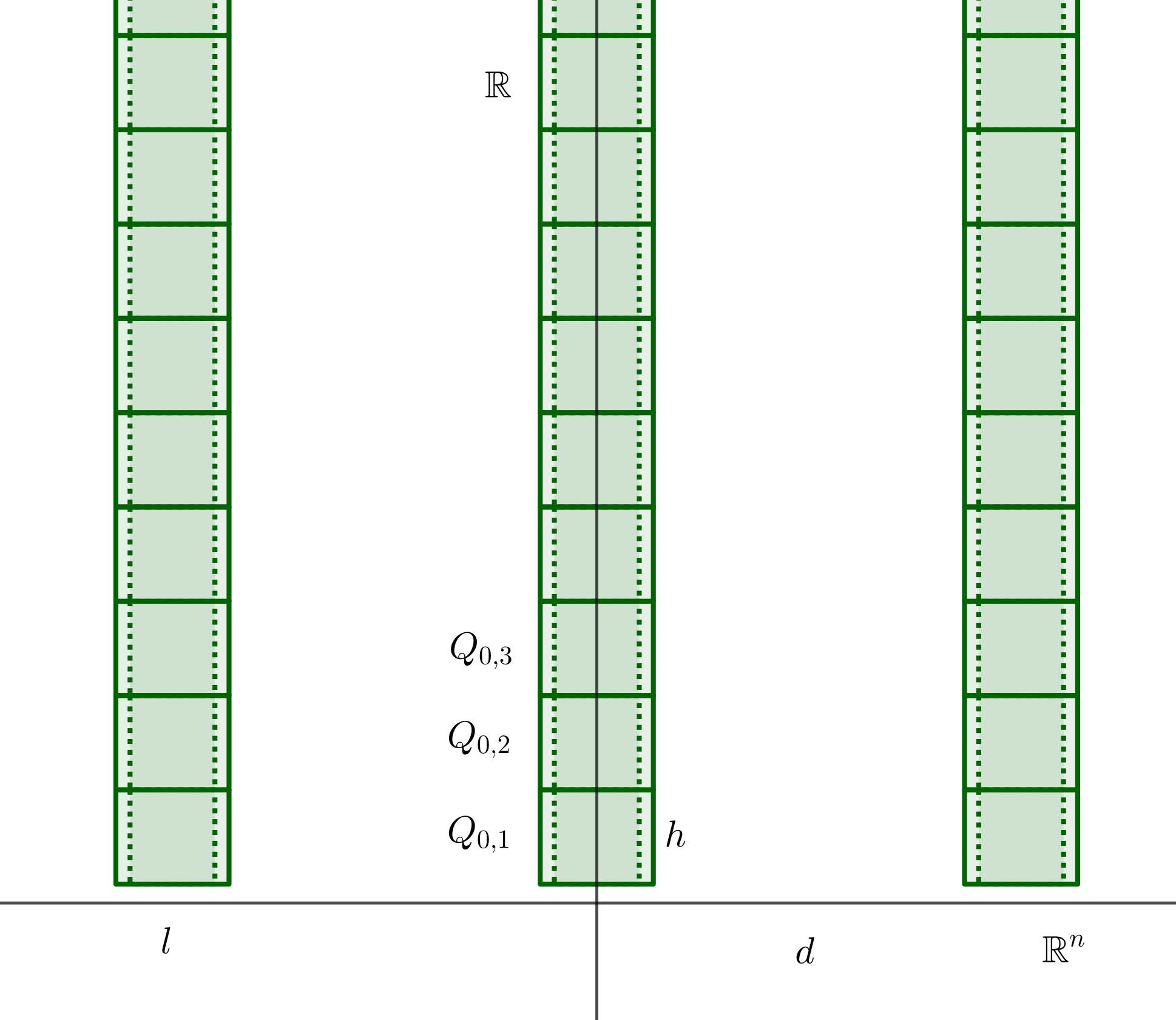}
\end{center}
\vspace{-5mm}
\caption{Decomposition. We decompose the upper half-space into cuboids $ Q_{a,j} $ of volume $ l^{n} h $ with stripes of the width $d$ between them. The centers of obstacles should lie in the smaller cuboids of volume $ (l-2 r_{1})^{n} h $, denoted by the dotted lines.}
\label{fig:decomposition}
\end{figure}

Now, for each $ a = ( a_{1} , \ldots , a_{n} ) \in \Z^{n} $ and $ j \in \N $, define 
\[ Q_{a} := \prod_{ i = 1 }^{n} \big[ a_{i} ( l + d ) - \tfrac{l}{2} , a_{i} ( l + d ) + \tfrac{l}{2} \big], \quad
Q_{a,j} := Q_{a} \times [ ( j - 1 ) h + r_{1} , j h + r_{1} ] \]
for some (still arbitrary) $ l > 2 r_{1} $ and $ h , d > 0 $.
We say that a point $ (a,j) $ is open if there exists a sufficiently strong obstacle 
such that its center $ (x,y) \in \R^{n} \times [ r_{1} , \i ) $ lies in $ Q_{a,j} $ 
and fulfils $ | x_{i} - a_{i} ( l + d ) | \le \frac{l}{2} - r_{1} $ for every $ i \in \{ 1 , \ldots , n \} $.
(It lies within the part of $ Q_{a,j} $ bounded by the dotted line in Figure~\ref{fig:decomposition}.)

The probability of the event that the center of an obstacle with strength at least $M$ lies in an cuboid of volume $ (l-2 r_{1})^{n} h $ is $ 1 - \exp [ - \lambda ( l - 2 r_{1} )^{n} h {\P}( f_{1} \ge M ) ] $. According to Theorem~\ref{theo:percolation} a box is thus open if 
\[ 1 - \exp [ - \lambda ( l - 2 r_{1} )^{n} h {\P}( f_{1} \ge M ) ] > 1 - \tfrac{1}{ (2n + 2)^{2} }, \]
%
%
%
%
%
%
%
%
%
or equivalently if
\begin{equation}
\label{eq:cond-lMh}
l > 2 r_{1} + \left( \frac{ 2 \log( 2 n + 2 ) }{ \lambda h {\P}( f_{1} \ge M ) } \right)^{1/n}. 
\end{equation}

The plan to construct a supersolution is to locate sufficiently many obstacles and use the flat supersolution adapted to the height of each such obstacle. This is done by adding a \emph{lifting function}, see also Figure~\ref{figure:lepilna-funkcija}.
\begin{prop}[{\cite[Proposition 2.13]{DDS}}]
\label{prop:lifting-function}
Let $ h, l, d > 0 $. Suppose $ y: \Z^{n} \to \R $ has the following property:
\[ \forall a,b \in \Z^{n}: \| a - b \|_{1} = 1 \ \Rightarrow \ | y(a) - y(b) | < 2h. \] 
Then there exists $ C_{1} = C_{1}(n) > 0 $ and a smooth function $ v_{\rm lift} : \R^{n} \to \R $ such that
\begin{itemize}
\item 
$ v_{\rm lift}|_{ Q_{a} } = y(a) $ for every $ a \in \Z^{n} $,
\item
$ \| D^{2} v_{\rm lift}  \|_{\i} \le C_{1} \frac{h}{ d^{2} } $,
\item
$ \| \nabla v_{\rm lift}  \|_{\i} \le C_{1} \frac{h}{d} $.
\end{itemize}
\end{prop}
%
\begin{figure}[ht]
\begin{center}
\includegraphics[width=\textwidth]{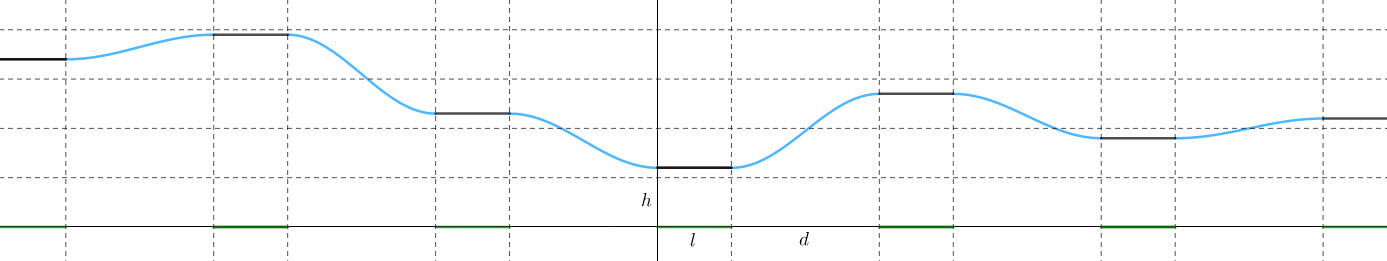}
\end{center}
\vspace{-5mm}
\caption{Lifting function. This is a smooth function that has the prescribed value in the boxes and for which we have bounds for the first and second derivative outside the boxes.}
\label{figure:lepilna-funkcija}
\end{figure}
%
%

Now we are in the position to state and prove the main result for the continuous case.
\begin{theo}
\label{theo:main-cont}
If
\[ \limsup_{ x \to \i } x^{ \frac{1}{2} + \frac{1}{n} }  {\P}( f_{1} \ge x )  = \i, \]
then infinite pinning occurs.
\end{theo}
\begin{remark}
In dimension 1, this condition can be fulfilled by distributions with finite expectation, e.g., by a 
Pareto distribution $ P(I)( 1 , \alpha) $ for some $1<\alpha< \frac{3}{2}$.
\end{remark}
%

\begin{proof}
Let $ K > 0 $ be arbitrary. 
According to our assumption, there exists $ M \ge 2 K $ such that 
\[ M^{ \frac{1}{2} + \frac{1}{n} } {\P}( f_{1} \ge M ) \ge K. \]
By choosing
\begin{equation}
l := 2 r_{1} + C_{0} \sqrt[n]{ \frac{ M^{ \frac{1}{2} + \frac{1}{n} } }{ h K } } 
\quad \mbox{with} \quad
C_{0} := \sqrt[n]{ \frac{ 3 \log( 2 n + 2 ) }{ \lambda } }
\end{equation}
(with $h>0$ still arbitrary), the condition~\eqref{eq:cond-lMh} is fulfilled.
Hence, by Theorem~\ref{theo:percolation} there exist an array of percolating open boxes $ ( Q_{ a , j_{a} } )_{ a \in \Z^{n} } $.
Let us denote the centers of corresponding obstacles by $ ( x_{a} , y_{a} ) $. 

Now we choose a local supersolution.
We take $ m \ge \max \{ n , 2 \} $ such that
%
%
\[ \frac{ M r_{0} }{2} \ge (m+n)(m+2) \ge \frac{ M r_{0} }{4} \]
(if necessary we take a larger $M$ at the start so that such $m$ exists),
and set the radii
\begin{equation*} 
r_{ \rm in } 
:= r_{0}, 
\quad
r_{ \rm out } 
:= \sqrt{n} \left( l + \frac{d}{2} - r_{1} \right) 
= \sqrt{n} \left( C_{0} \sqrt[n]{ \frac{ M^{ \frac{1}{2} + \frac{1}{n} } }{ h K } }
	+ \frac{d}{2} + r_{1} \right)
\end{equation*}
and the forces
\begin{equation}
\label{eq:F-in-out}
F_{ \rm in } := \frac{ (m+n)(m+2) }{ r_{0} },
\quad
F_{ \rm out } := 2 C_{1} \frac{ h }{ d^{2} } 
\end{equation}
with $ C_{1} $ as Proposition~\ref{prop:lifting-function} and $ h,d > 0 $ still free.
We denote the corresponding solution by $ v_{ \rm local } $.
In order to be a local supersolution, it must fulfil inequality \eqref{eq:cond-loc-subsol}, which reads for our choice of $ F_{\rm in} $ and $ F_{\rm out} $
\begin{equation}
\label{eq:ineq}
m + 2 \ge 2 C_{1} \frac{ r_{0} }{n} \frac{ h }{ d^{2} } \left( - 1 + \frac{ r_{ \rm out }^{n} }{ r_{0}^{n} } \right). 
\end{equation}
Since for $ m \ge \max \{ n , 2 \} $ it holds $ M \le \frac{16}{ r_{0} } m^{2} $, and assuming $ d \ge 2 r_{1} $,
we may estimate
\[ 2 C_{1} \frac{ r_{ \rm out }^{n} }{ n r_{0}^{n-1} } 
\le 2 C_{1} \frac{ n^{n/2} 2^{n-1} }{ n r_{0}^{n-1} } 
		\left( C_{0}^{n} \frac{ M^{ \frac{1}{2} + \frac{1}{n} } }{h K}
		+ \left( \frac{d}{2} + r_{1} \right)^{n} \right)
\le C_{2} \left( \frac{ m^{ 1 + \frac{2}{n} } }{h K} + d^{n} \right) \]
with $ C_{2} = C_{2}( n , \lambda , r_{0} ) $.
Inequality~\eqref{eq:ineq} will surely hold if the following is fulfilled:
\[ m 
\ge C_{2} \frac{ h }{ d^{2} } \left( \frac{ m^{ 1 + \frac{2}{n} } }{hK} + d^{n} \right) 
= C_{2} \left( \frac{ m^{ 1 + \frac{2}{n} } }{ d^{2} K } +  h d^{n-2} \right).
\]
Let us simply fulfil this condition by setting both summands at $ \frac{m}{2} $.
Hence, we define 
\begin{equation*}
d := \sqrt{ \frac{ 2 C_{2} }{K} } m^{ \frac{1}{n} }
\end{equation*}
(where, if necessary, we take appropriate larger $M$ and $m$ so that $ d \ge 2 r_{1} $) and
\begin{equation*}
h := K^{ \frac{n}{2} - 1 } ( 2 C_{2} )^{ - \frac{n}{2} } m^{ \frac{2}{n} }.
\end{equation*}
Now all the scales are set, and $ v_{ \rm local } $ is a local (viscosity) supersolution.

We chose $ r_{ \rm out } $ sufficiently large so that the domain of the flat supersolution
\[ v_{ \rm flat }(x) := \min_{ a \in \Z^{n} } v_{ \rm local }( x - x_{a} ) \]
is $ \R^{n} $. 
We take $ v_{ \rm lift } $ as in Proposition~\ref{prop:lifting-function} with $ y(a) := y_{a} $ for each $ a \in \Z^{n} $.

Since $ v_{\rm local} $ suffices \eqref{eq:local-sol} and $ \| \Delta v_{\rm lift} \|_{\i} \le C_{1} \frac{h}{d^{2}} $, 
the function $ v := v_{ \rm flat } + v_{ \rm lift } $ satisfies
\[ 0 \ge \Delta v( x , \omega ) - f( x , v( x , \omega ) , \omega ) + F \]
for any $ 0 < F \le \min \{ F_{ \rm out } - C_{1} \frac{ h }{ d^{2 } } , M - F_{\rm in} \} $.
By \eqref{eq:F-in-out} 
\[ M - F_{\rm in} \ge \frac{M}{2} \ge K 
\quad \mbox{and} \quad
F_{ \rm out } - C_{1} \frac{ h }{ d^{2 } }
= C_{1} \frac{ h }{ d^{2 } } 
= \frac{ C_{1} }{ ( 2 C_{2} )^{ \frac{n}{2} + 1 } } K^{ \frac{n}{2} }. \]

To conclude, for a given $F$ we choose $ K \ge F $ such that also
\[ F \le \frac{ C_{1} }{ ( 2 C_{2} )^{ \frac{n}{2} + 1 } } K^{ \frac{n}{2} } \] 
and make the construction above.
Hence, for any $F$ pinning takes place.
\end{proof}

\begin{remark}
A sufficient condition for Theorem~\ref{theo:main-cont} to hold is that $  \E( f_{1}^b ) = \i  $ for some $ 0 < b < \frac{1}{2} + \frac{1}{n} $.
This is a consequence of the following lemma.
\end{remark}

\begin{lemma}
Let $ a>b>0 $ and suppose $ X \ge 0 $ is a real-valued random variable with $ \E ( X^b )=\infty $. Then $ \limsup_{x \to \infty} \P(X \ge x) x^a =\infty $.
\end{lemma}

\begin{proof}
We may rewrite the assumption as
\[ \infty=\E( X^b )=\int_0^\infty \P(X^b \ge x) \x = \int_0^\infty \P(X \ge x^{1/b}) \x. \]
If the claim were wrong, there would exist a $C <\infty$ such that $ \P(X \ge x) x^a \le C $ for all $ x \ge 0 $ which would lead to contradiction as
\[ \E( X^b ) \le 1 +  \int_1^\infty \P(X \ge x^{1/b}) \x \le 1+C \int_1^\infty  x^{-a/b}  \x <\infty. \]
\end{proof}

The two conditions are actually very close.

\begin{lemma}
Let $ X \ge 0 $ be a real-valued random variable. If for some $ b > 0 $ it holds that $\limsup_{x \to \infty} x^b\P(X \ge x) =\infty$, then $\E (X^b)=\infty$.
\end{lemma}
\begin{proof}
For every $ M>0 $ there exists a $ y>0 $ such that $ \P(X \ge y) \ge M  y^{-b}$. Therefore,
\begin{align*}
\E (X^b) 
& =\int_0^\infty \P(X^b \ge x) \x = \int_0^\infty \P(X \ge x^{1/b}) \x \ge \\
& \ge \int_0^{y^b} \P(X \ge x^{1/b}) \x \ge y^b \P(X \ge y)\ge y^b M y^{-b}=M.
\end{align*}
Since $M$ was arbitrary, the claim follows.
\end{proof}

\section{Conclusion}
We have shown that in our discrete setting in one space dimension, infinite pinning arises if the random, independent obstacles' strengths have infinite second moment. This should be compared to the continuous case, where infinite pinning in one space dimension occurs if for some $p<\frac{3}{2}$, the $p$-th moment of the obstacles' strengths is infinite. This difference mainly seems to arise due to the fact that not the full obstacle strength can be used in continuum models, as otherwise the pinned interface would not wholly lie inside an obstacle. The question of infinite pinning in discrete models in more than one space dimension remains open. The main question is, however, whether the second moment condition is indeed also necessary for infinite pinning, as the absence of infinite pinning has, to this date, only been shown for models with bounded exponential moment \cite{MR2740531,MR2908614,MR3719956,Bodineau:2013ur}.


\bibliographystyle{abbrv}
\bibliography{DJS-infinite_pinning}

\begin{thebibliography}{1}

\bibitem{Bodineau:2013ur}
T.~Bodineau and A.~Teixeira.
\newblock {Interface Motion in Random Media}.
\newblock {\em Communications in Mathematical Physics}, 334(2):843--865, Mar.
  2015.

\bibitem{MR2740531}
J.~Coville, N.~Dirr, and S.~Luckhaus.
\newblock Non-existence of positive stationary solutions for a class of
  semi-linear {PDE}s with random coefficients.
\newblock {\em Netw. Heterog. Media}, 5(4):745--763, 2010.

\bibitem{Crandall:1992kn}
M.~G. Crandall, H.~Ishii, and P.-L. Lions.
\newblock {User's guide to viscosity solutions of second order partial
  differential equations}.
\newblock {\em American Mathematical Society. Bulletin. New Series},
  27(1):1--67, 1992.

\bibitem{DDGHS}
N.~Dirr, P.~W. Dondl, G.~R. Grimmett, A.~E. Holroyd, and M.~Scheutzow.
\newblock Lipschitz percolation.
\newblock {\em Electron. Commun. Probab.}, 15:14--21, 2010.

\bibitem{DDS}
N.~Dirr, P.~W. Dondl, and M.~Scheutzow.
\newblock Pinning of interfaces in random media.
\newblock {\em Interfaces Free Bound.}, 13(3):411--421, 2011.

\bibitem{DJS}
P.~W. Dondl, M.~Jesenko, and M.~Scheutzow.
\newblock Pinning of interfaces in a random medium with zero mean.
\newblock {\em arXiv:2002.00800 [math.AP]}, 2020.

\bibitem{MR2908614}
P.~W. Dondl and M.~Scheutzow.
\newblock Positive speed of propagation in a semilinear parabolic interface
  model with unbounded random coefficients.
\newblock {\em Netw. Heterog. Media}, 7(1):137--150, 2012.

\bibitem{MR3719956}
P.~W. Dondl and M.~Scheutzow.
\newblock Ballistic and sub-ballistic motion of interfaces in a field of random
  obstacles.
\newblock {\em Ann. Appl. Probab.}, 27(5):3189--3200, 2017.

\end{thebibliography}


\end{document}